\documentclass[11pt]{article}

\usepackage{amsmath,amsthm,amssymb}
\usepackage[usenames,dvipsnames]{xcolor}
\usepackage{enumerate}
\usepackage{graphicx}
\usepackage{cite}
\usepackage{comment}
\usepackage[margin=1in]{geometry}

\usepackage[pdftitle={Closure of LQG},
  pdfauthor={Andres Contreras-Hip and Ewain Gwynne},
colorlinks=true,linkcolor=NavyBlue,urlcolor=RoyalBlue,citecolor=PineGreen,bookmarks=true,bookmarksopen=true,bookmarksopenlevel=2,unicode=true,linktocpage]{hyperref}

\theoremstyle{plain}
\newtheorem{thm}{Theorem}[section]
\newtheorem{theorem}[thm]{Theorem}

\newtheorem{lemma}[thm]{Lemma}

\def\@rst #1 #2other{#1}
\newcommand\MR[1]{\relax\ifhmode\unskip\spacefactor3000 \space\fi
  \MRhref{\expandafter\@rst #1 other}{#1}}
\newcommand{\MRhref}[2]{\href{http://www.ams.org/mathscinet-getitem?mr=#1}{MR#2}}

\theoremstyle{definition}
\newtheorem{defn}[thm]{Definition}

\numberwithin{equation}{section} 

\newcommand{\dsb}{\begin{adjustwidth}{2.5em}{0pt}
\begin{footnotesize}}
\newcommand{\dse}{\end{footnotesize}
\end{adjustwidth}}

\newcommand{\ssb}{\begin{adjustwidth}{2.5em}{0pt}}
\newcommand{\sse}{\end{adjustwidth}}

\newcommand{\aryb}{\begin{eqnarray*}}
\newcommand{\arye}{\end{eqnarray*}}
\def\alb#1\ale{\begin{align*}#1\end{align*}}
\def\allb#1\alle{\begin{align}#1\end{align}}
\newcommand{\eqb}{\begin{equation}}
\newcommand{\eqe}{\end{equation}}
\newcommand{\eqbn}{\begin{equation*}}
\newcommand{\eqen}{\end{equation*}}

\newcommand\p{\partial}
\newcommand\e{\varepsilon}

\newcommand\R{\mathbb{R}}
\newcommand\Z{\mathbb{Z}}
\newcommand\N{\mathbb{N}}

\newcommand\norm[1]{\lVert#1\rVert}

\newcommand\vphi{\varphi}

\let\originalleft\left
\let\originalright\right
\renewcommand{\left}{\mathopen{}\mathclose\bgroup\originalleft}
\renewcommand{\right}{\aftergroup\egroup\originalright}

\title{Approximation of length metrics by conformally flat Riemannian metrics}
 \date{ }
 \author{ 
\begin{tabular}{c} Andres A. Contreras Hip and Ewain Gwynne\\ \small University of Chicago \end{tabular}  
}

\begin{document}

\maketitle

\begin{abstract}
We present a proof of the folklore result that any length metric on $\mathbb R^d$ can be approximated by conformally flat Riemannian distance functions in the uniform distance. This result is used to study Liouville quantum gravity in another paper by the same authors.  
\end{abstract}

\tableofcontents

\section{Introduction}\label{appendix}
In this note we show that any length metric space homeomorphic to $\mathbb{R}^d$ can be approximated by conformally flat Riemannian metrics on $\mathbb{R}^d.$ More precisely, let $D_0$ denote the Euclidean metric on $\R^d,$ and define $e^f \cdot D_0$ to be the distance function associated with the Riemannian metric $e^f (dx_1^2+ \ldots + dx_d^2).$ Before we state the theorem, we will need the following definition.
\begin{defn} \label{def-length}
Let $(X,\mathfrak d)$ be a metric space. 
For a path $P : [a,b]\to X$, we define the \emph{$\mathfrak d$-length} of $P$ by
\eqbn
\ell_{\mathfrak{d}}(P) := \sup_{T} \sum_{i=1}^{\# T} \mathfrak d(P(t_i) , P(t_{i-1})) 
\eqen
where the supremum is over all partitions $T : a= t_0 < \dots < t_{\# T} = b$ of $[a,b]$.  
We say that $\mathfrak d$ is a \textbf{length metric} if for each $x,y\in X$, the distance $\mathfrak d(x,y)$ is equal to the infimum of the $\mathfrak d$-lengths of the $\mathfrak d$-continuous paths from $x$ to $y$. A path from $x$ to $y$ of minimal $\mathfrak d$-length is called a \textbf{geodesic}.
\end{defn}

We say that $(X,\mathfrak d)$ is \textbf{boundedly compact} is each closed bounded subset of $X$ is $\mathfrak d$-compact. 
It is shown in~\cite[Corollary 2.5.20]{bbi-metric-geometry} that if $\mathfrak d$ is a length metric and $(X,\mathfrak d)$ is boundedly compact, then there is a $\mathfrak d$-geodesic between any two points of $X$.

\begin{theorem}\label{app}
Let $\bar{D}$ be a boundedly compact length metric on $\R^d$ which induces the same topology as the Euclidean metric $D_0.$ Let $\e>0$ and $R>0.$ Then there exists a bounded continuous function $f:\R^d\to \R$ such that
\[
\sup_{x,y \in [-R,R]^d}\vert \bar{D}(x,y)-e^f\cdot D_0(x,y) \vert \leq \e.
\]
\end{theorem}
\noindent This result is stated in dimension $2$ without proof in \cite{bbi-metric-geometry} right after Exercise 7.1.3. 

Recall that a Riemannian metric $g$ on $\mathbb R^d$ is \emph{conformally flat} if $g = e^f (dx_1^2 + \dots + dx_d^2)$ for some function $f : \mathbb R^d \to \mathbb R$.
In dimension $2,$ every Riemannian metric is locally conformally flat. In higher dimensions this is no longer true: conformal flatness is equivalent to the vanishing of the Weyl tensor in dimension $3$ or the Cotton tensor for $d \geq 4.$ Nevertheless, we show that conformally flat metrics suffice to approximate any length metric, including ones which arise from a Riemannian metric which is not conformally flat and ones which do not arise from a Riemannian metric. 

Our result is used in~\cite{cg-support-thm} to show that any boundedly compact length metric on $\mathbb R^d$ can be approximated with positive probability by the random metrics arising from Liouville quantum gravity and its higher-dimensional analogs.

To prove Theorem~\ref{app}, we proceed as follows. We approximate our length metric by a weighted graph embedded in $\mathbb R^d$ by considering geodesics between points in a fine grid (see \eqref{dGdclose}). These weights are chosen in such a way that the graph distance metric $d_G$ coincides with $\bar{D}$ for pairs of vertices. After this, we define an appropriate function $f$ which is very large far away from the graph's edges, and then we define $f$ according to the weights near each edge (see \eqref{fdef}). These weights are chosen in such a way that the metrics $e^f\cdot D_0$ and $d_G$ are close for vertices. Additionally, $f$ is chosen to be large enough away from the edges of $G$ in such a manner that $e^f\cdot D_0$ geodesics are forced to stay near $G$'s edges, but at the same time not too large, so that points away from the vertices of $G$ are still at a small distance from the edges (see Lemma \ref{triptohighway}). After this, we show that the metrics $d_G$ and $e^f\cdot D_0$ are close for adjacent vertices (Lemmas \ref{upperedge} and \ref{loweredge}) and later show that $e^f\cdot D_0$ distance minimizing paths between points in a neighborhood of the graph $G$ remain in a slightly thicker neighborhood of $G$ (see Lemma \ref{trapped}). Combining these lemmas together with with the fact that $d_G$ approximates $\bar{D}$ we conclude.

\section*{Acknowledgements}

We thank Dmitri Burago for helpful discussions. E.G.\ was partially supported by a Clay research fellowship and by National Science Foundation grant DMS-2245832. 

\section{Proof of Theorem \ref{app}}

Let $\e>0.$ We will first construct a graph embedded into $\R^d$ with weights on the edges, with the property that the weighted graph distance between any two vertices is close to their $\bar{D}$-distance. The edges of the graph will be approximations of $\bar{D}$-geodesics. After this, we choose an appropriate function $f$ such that $e^f \cdot D_0$ approximates this graph metric. Finally, we put these together and bound distances between any pair of points in $\R^d.$
 
We start by constructing the graph approximation of $\bar{D}$.

\subsection{Construction of an approximating graph}

Let $\bar{n}\in\N,$ and subdivide $[-R,R]^d$ into $\bar{n}^d$ cubes $S_{i_1i_2\cdots i_d}$ for $i_1,i_2,\ldots ,i_d\in\{1,\dots,\bar n\}$ of side length $\frac{R}{\bar{n}}.$ Again, since $\bar{D}$ and $D_0$ induce the same topology, we can take $\bar{n}$ large enough so that
\begin{equation}\label{smalldiam}
\sup_{i_1,i_2,\cdots ,i_d} \mathrm{Diam}_d (S_{i_1,i_2,\cdots ,i_d}+B_{R/(2\bar{n})}(0)) \leq \frac{\e}{64}.
\end{equation}
Take a large integer $\bar{m},$ and further subdivide each $S_{i_1\cdots i_d}$ into $\bar{m}^d$ cubes of side $\frac{R}{\bar{n}\bar{m}}.$ We label these cubes as $\{\bar S_{j_1 \cdots j_d}^{i_1\cdots i_d} : i_1,\ldots ,i_d\in \{1,\dots,\bar n\}, j_1, \ldots , j_d \in \{1,\dots,\bar m\}\}$ in such a way that
\[
\bigcup_{j_1, \ldots , j_d} \bar{S}_{j_1 \cdots j_d}^{i_1\cdots i_d} = S_{i_1\cdots i_d}.
\]
Let $G_1$ denote the graph consisting of the vertices $\frac{R}{\bar n \bar m} \Z^d \cap [-R,R]^d$ and with edges being $\bar{D}$-length minimizing paths between all pairs of vertices. Recall from just after Definition~\ref{def-length} that such paths exist since $\bar{D}$ is a boundedly compact length metric. The geodesics in this set can intersect each other, possibly infinitely many times.

Our next step is to replace these geodesics by other $\bar{D}$-length minimizing geodesics such that any pair intersects at either the endpoints, or at a single curve segment. More precisely:
\begin{lemma}
Suppose that $\mathcal{S}$ is a finite set of $\bar{D}$-geodesics between pairs of points in $\R^d.$ Then there exists another set $\bar{\mathcal{S}}$ of geodesics between the same pairs of points such that each the intersection of each pair of geodesics in $\bar{\mathcal{S}}$ is either a common segment or a single point.
\end{lemma}
\begin{proof}
We first consider a single pair of geodesics, and note that we can replace these two geodesics so that they intersect at either a single point or in an interval. Indeed, suppose $\gamma_1,\gamma_2:[0,1]\to \R^d$ are geodesics between $x_1,x_2$ and $y_1,y_2$ respectively and that they intersect at at least two points. Let 
	\[
	t_1:=\inf\{\sigma \in [0,1]: \gamma_1(\sigma) \in \gamma_2([0,1])\} ,
	\]
	\[
	t_2:=\sup\{\sigma \in [0,1]: \gamma_1(\sigma) \in \gamma_2([0,1])\} ,
	\]
and let $s_1,s_2$ such that $p_1:=\gamma_1(t_1)=\gamma_2(s_1),$ and $p_2:=\gamma_1(t_2)=\gamma_2(s_2).$ Recall that since $\gamma_1,\gamma_2$ intersect at least twice, $p_1\neq p_2.$ Since $\gamma_2$ is a $\bar{D}$-length minimizing geodesic, it is also length minimizing when restricting to an interval. Therefore $\gamma_2\vert_{[s_1,s_2]}$ is a length minimizing path between $p_1,p_2.$ Thus we can replace $\gamma_1$ by $\tilde{\gamma}_1=\gamma_1\vert_{[0,t_1]} \ast \gamma_2\vert_{[s_1,s_2]}\ast\gamma_1\vert_{[t_2,1]},$ where $\ast$ denotes concatenation. This procedure lets us replace two geodesics with multiple intersection points with geodesics whose intersection is a single interval (see Figures \ref{georeppic1} and \ref{georeppic2}). Suppose we number the geodesics in $\mathcal{S}$ arbitrarily, so that the geodesics are the set $\{P_k\}_{k=1}^n$ for some large $n.$ We will inductively apply this procedure to to get rid of all intersections that are not at endpoints or single segments. First we apply the procedure to $P_1$ and $P_2$ to replace them with geodesics intersecting at their endpoints. Now suppose that we have replaced $\{P_k\}_{k=1}^m$ with geodesics intersecting only at endpoints, where $m<n.$ Let
\[
t_1:= \inf\{t:P_{m+1}(t)\in \cup_{j=1}^{m+1}P_j([0,1])\}.
\]
Now let $j_1$ be such that $P(t_1)\in P_{j_1}.$ Then we define
\[
t_1':=\sup\{t:P_{m+1}(t)\in P_{j_1}([0,1])\}.
\]
Suppose we have defined $t_1, \ldots t_i,$ $t_1',\ldots , t_i',$ and $j_1,\ldots j_i.$ Then define
\[
t_{i+1}:= \inf\{t\geq t_i':P_{m+1}(t)\in \cup_{j\notin\{j_1, \ldots , j_i\}}P_j([0,1])\},
\]
let $j_{i+1}$ such that $P(t_{i+1})\in P_{j_{i+1}}([0,1]),$ and finally let
\[
t_{i+1}':=\sup\{t:P_{m+1}(t)\in P_{j_{i+1}}([0,1])\}.
\]

 Note that the intervals $[t_i,t_i']$ are disjoint. Now applying the above procedure we can replace $P_{m+1}$ in the intervals $[t_i,t_i']$ by the corresponding segments on $P_{i+1},$ thus removing multiple intersections, leaving us with only intersection at the endpoints. This completes the proof.
\end{proof}

\begin{figure}
\centering
\begin{minipage}{0.4\textwidth}
  \centering
  \includegraphics[width=0.9\textwidth]{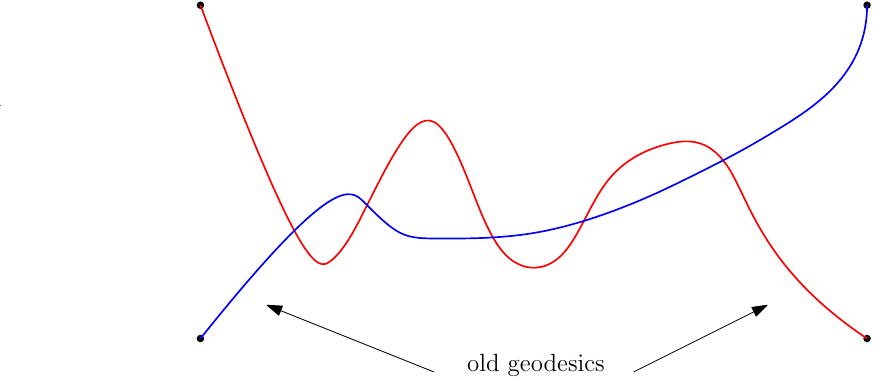}
  \caption{\label{georeppic1} old geodesics (red and blue paths)}
\end{minipage}
\begin{minipage}{0.4\textwidth}
  \centering
  \includegraphics[width=0.9\linewidth]{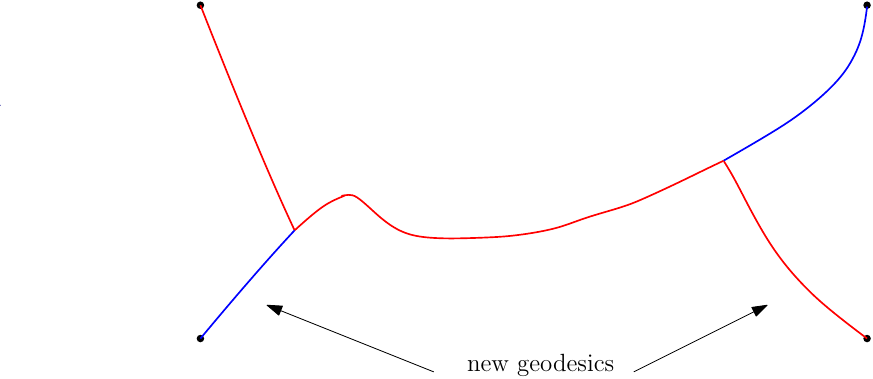}
  \caption{\label{georeppic2} New geodesics, cut from the red and blue geodesics}
\end{minipage}
\end{figure}

By applying the above lemma to the geodesics of $G_1$ we end up with a graph $G_2$ whose vertices are points in $\frac{R}{\bar n \bar m} \Z^d \cap [-R,R]^d$ and whose edges are geodesics between the vertices and intersect either at points or single curve segments. 

Now let $\bar{\eta},\zeta>0$ be small. Note that for large enough $\bar{n},\bar{m},$ any edge in the graph $G_2$ between two vertices in $\bar{S}_{j_1\cdots j_d}^{i_1\cdots i_d}$ is contained in $S_{i_1\cdots i_d}+B_{\zeta}(0).$ This comes from the uniform continuity of $\bar{D}$ with respect to $D_0.$
Now at the points where any pair of geodesics merge, add a vertex to the graph, and divide the edges accordingly to obtain a new graph $G_3.$ 

We will now replace each geodesic with a piecewise linear approximation. To ensure that these piecewise linear approximations do not intersect, we will use a larger step size near each vertex, and a smaller step size on each edge far away from vertices. Let $\tau>0$ be small, and let $K$ be a large integer (these are free variables that will be chosen later). First we subdivide each geodesic edge $e$ of $G_3$ into three segments: two at the endpoints of $\bar{D}$-length $\tau,$ and a middle segment of length $\ell_d(e)-2\tau.$ We further subdivide the middle segment into segments of length $\frac{\ell_d(e)-2\tau}{K}$ (see Figure \ref{segpic}). Let the points where these line segments meet be new vertices. We have now obtained a new graph $G_4.$ If $K$ is chosen large enough so that
\[
\frac{1}{K} < \inf_{\substack{e_1,e_2\\ \mathrm{edges}}}\bar{D}(e_1,e_2)
\]
and $\tau \leq \frac{1}{10}\min_{\substack{v_1,v_2\\ \mathrm{edges}}} d(v_1,v_2),$ it is easy to see that the new edges do not intersect.
Let
\[
E=\cup_{e\in \mathcal{E}G} e
\]
where $\mathcal{E} G$ denotes the edge set of $G,$ and for any $\eta>0$
\[
E_\eta:=\{z\in \R^d: \bar{D}(z,E)\leq \eta\}.
\]


\begin{figure}
\begin{center}
\includegraphics[width=0.5\textwidth]{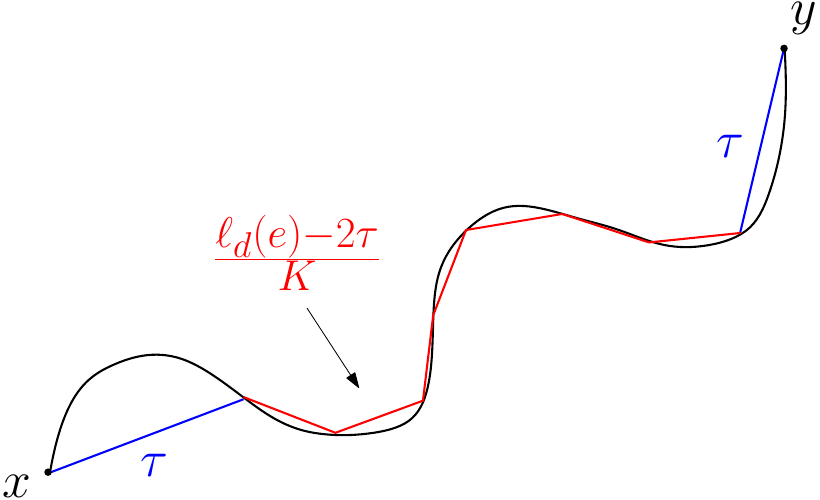}
\caption{\label{segpic} Approximation by line segments
}
\end{center}
\end{figure}

We will now assign weights to all edges of this graph. If $e=(x,y)$ denotes an edge on the graph $G_4$ with $x,y$ as its endpoints, then let $w_e:=\bar{D}(x,y)$ be the corresponding weights and define $d_G$ to be the weighted graph distance on $G_4.$ Let $G$ denote this weighted graph we have constructed, and let $\mathcal{V}G$ denote the set of vertices of $G.$ Let $\delta > 0.$ Then taking $\bar{n},\bar{m}$ large enough and recalling that $\bar{D}$ is uniformly continuous with respect to $D_0$ we see that $G$ is a weighted graph whose vertices form a $\delta$-net with respect to both the $\bar{D}$ and $D_0$ metrics, whose edges are straight line segments, and whose edges only intersect at endpoints.


\begin{figure}
\centering
\begin{minipage}{0.4\textwidth}
  \centering
  \includegraphics[width=0.9\textwidth]{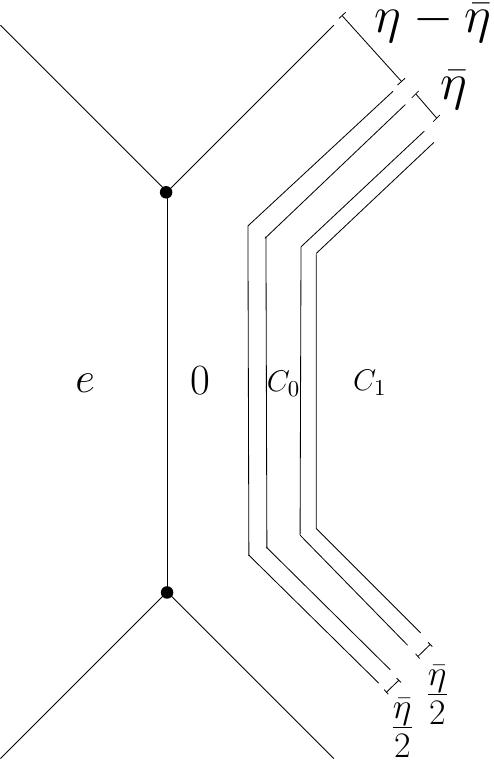}
\end{minipage}
\begin{minipage}{0.4\textwidth}
  \centering
  \includegraphics[width=0.9\linewidth]{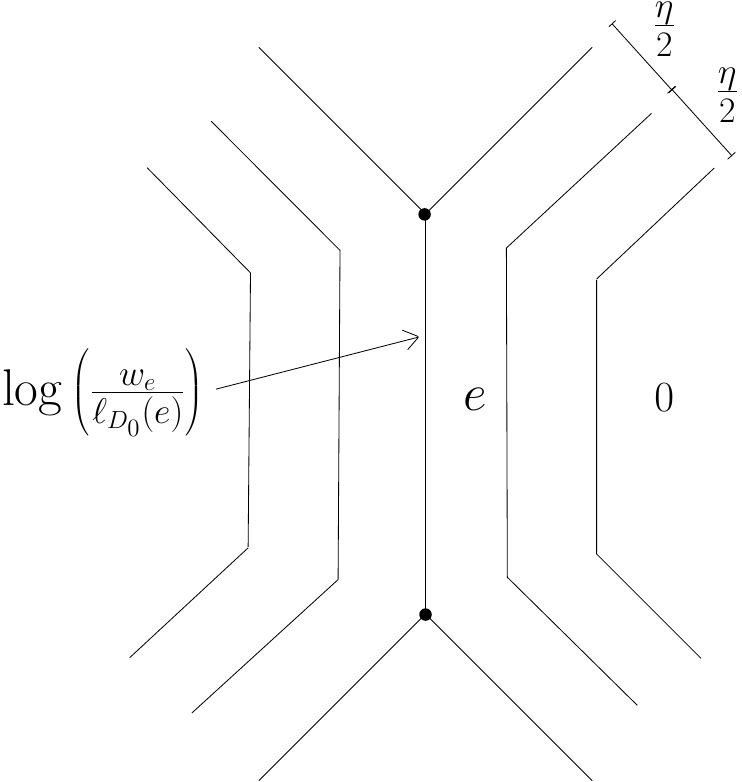}
\end{minipage}
\caption{\label{fFetapic} {\bf Left:} Values of $F_{\eta,\bar{\eta}}^{\mathrm{ext}}$. {\bf Right:} Values of $f_{e,\eta}$
}
\end{figure}

Note that 
\begin{equation}\label{dGdclose}
\sup_{\substack{x,y\in G\\
\mathrm{vertices}}} \vert \bar{D}(x,y)-d_G(x,y)\vert \leq\frac{\e}{4}
\end{equation}
for large enough $n,m.$ Indeed, let $x_1, \ldots ,x_n$ be such that the path formed by the edges $(x_i,x_{i+1})$ is the path of least $d_G$ length between $x,y.$ Then since
\[
\bar{D}(x_i,x_{i+1}) = d_G(x_i,x_{i+1})
\]
we see that 
\[
\bar{D}(x,y)\leq d_G(x,y).
\]
For the reverse inequality one uses the fact that the set of vertices of $G$ is an $\e$-net together with the fact that the edges of $G_3$ are $\bar{D}$-geodesics.

Let $\eta>0,$ and define $\mathcal{F}_\eta$ by
\[
\mathcal{F}_\eta = \{z \in \R^d: D_0(z,E)=\eta\}
\]
Also let $C_0,C_1>0$ be large constants to be chosen later. Let $\bar{\eta} \in (0,\eta/100)$ be a small constant. Then define a smooth function $F_{\eta,\bar{\eta}}^{\mathrm{ext}}$ such that 
\[
F_{\eta,\bar{\eta}}^{\mathrm{ext}}(z)=
\left\{
\begin{matrix}
C_0, & \mbox{ if } D_0(z,\mathcal{F}_\eta) \leq \frac{\bar{\eta}}{2}\\
C_1, & \mbox{ if } z \in \R^d \setminus E_{\eta+\bar{\eta}}\\
0, & \mbox{ if } z \in E_{\eta-\bar{\eta}}
\end{matrix}
\right.
\]
where $D_0$ denotes the Euclidean distance, and such that
\[
\min\{C_0,C_1\}\leq F_{\eta,\bar{\eta}}^{\mathrm{ext}}(z)\leq \max\{C_0,C_1\},
\]
if $\frac{\bar{\eta}}{2}\leq D_0(z,\mathcal{F}_\eta)\leq \bar{\eta},$ and
\[
0\leq F_{\eta,\bar{\eta}}^{\mathrm{ext}}(z)\leq C_0,
\]
if $z\in E_{\eta+\frac{\bar{\eta}}{2}}\setminus E_{\eta-\bar{\eta}}.$
For each edge $e \in E$ we define a smooth function $f_e(z)$ such that  
\[
f_{e,\eta}(z):= 
\left\{
\begin{matrix}
\log \left(\frac{w_e}{\ell_{D_0}(e)} \right), & \mbox{ if } D_0(z,e) \leq \frac{\eta}{2},\\
0, & \mbox{ if } D_0(z,e) \geq \eta,
\end{matrix}
\right.
\]
and $0\leq f_{e,\eta}(z)\leq \log \left(\frac{w_e}{\ell_{D_0}(e)} \right)$ if $\frac{\bar{\eta}}{2} \leq D_0(z,e) \leq \bar{\eta}.$ See Figure \ref{fFetapic} for pictures of $F_{\eta,\bar{\eta}}^{\mathrm{ext}},f_{e,\eta}.$
Define
\begin{equation}\label{fdef}
f:= \sum_e f_{e,\eta} + F_{\eta,\bar{\eta}}^{\mathrm{ext}}
\end{equation}
We claim that
\begin{equation}\label{closemetrics}
\sup_{x,y \in [-R,R]^d} \vert \bar{D}(x,y)-e^f\cdot D_0(x,y) \vert \leq \frac{\e}{2}
\end{equation}
for an appropriate choice of $C_0,C_1$ and $\bar{n},\bar{m}.$ Now we will choose these parameters. First, since $\bar{D}$ generates the Euclidean topology, there is an increasing function $\vphi:\R^+ \to \R^+$ such that $\lim_{\ell\to 0}\vphi(\ell)=0$ and such that
\begin{equation}\label{vphidefdef}
\sup_{x,y\in [-R,R]^d:D_0(x,y)\leq \ell} \bar{D}(x,y) \leq \vphi(\ell).
\end{equation}
Also, let $\Psi:\R^+\to\R^+$ be a function such that if $\vert z-w\vert\leq \Psi(\theta),$ then $\bar{D}(z,w)\leq\theta.$
Let $\bar{n},\bar{m}$ be such that
\begin{equation}\label{preparam1}
\bar{n}\bar{m} \geq \frac{1}{32\e}\sup_{\Psi(\e) \leq \ell \leq 2R} \frac{\vphi(\ell)}{\ell},
\end{equation}
and
\begin{equation}\label{smallwe}
w_e \leq \frac{\e}{128}.
\end{equation}
Let $\eta$ be small enough so that
\begin{equation}\label{etasmol}
\eta< \inf_{e}\frac{\e\ell_{D_0}(e)}{512 w_e},
\end{equation}
where the infimum is taken over edges of $G$ and let
\[
\bar{\eta}< \min\{\eta, \frac{1}{4}\Psi(\e)\}
\]
so that combining with \eqref{preparam1} we obtain
\begin{equation}\label{param1}
\bar{n}\bar{m} \geq \frac{1}{64\e}\sup_{\Psi(\e) \leq \ell \leq 2R} \frac{\vphi(\ell)}{\ell-\bar{\eta}},
\end{equation}

Now let
\begin{equation}\label{param2}
C_0=\log\left(\frac{\e}{128\bar{\eta}}\right),
\end{equation}
\begin{equation}\label{param3}
C_1 = \log \left(\frac{\e\bar{n}\bar{m}}{64}\right).
\end{equation}
We note that with the above choices we have
\begin{equation}\label{C0C1}
C_0>C_1.
\end{equation}
Also note that \eqref{closemetrics} implies Theorem \ref{app}.
To prove \eqref{closemetrics} we will need a couple of lemmas:
\begin{lemma}\label{graphapproxlowerbound}
Let $v_1,v_2 \in G,$ be adjacent vertices. Let $f$ be as in \eqref{fdef}. For small enough $\eta, \bar\eta$ we have
\[
d_G(v_1,v_2) - \frac{\e}{10 \vert G\vert} \leq e^f\cdot D_0(v_1,v_2)
\]
where $\vert G\vert$ denotes the number of edges in $G.$
\end{lemma}
\begin{lemma}\label{graphapproxupperbound}
Let $v_1,v_2 \in G,$ be adjacent vertices. Let $f$ be as in \eqref{fdef}. For small enough $\eta, \bar\eta$ we have
\[
e^f\cdot D_0(v_1,v_2) \leq d_G(v_1,v_2),
\]
where $\vert G\vert$ denotes the number of edges in $G.$
\end{lemma}
These two lemmas tell us that for adjacent vertices, the distances $e^f\cdot D_0$ and $d_G$ are very close. Next we need a lemma that constrains geodesics between points close to edges of $G$ to also be close to edges of $G.$
\begin{lemma}\label{trapped}
Let $\bar{\eta}>0.$ Suppose that $x,y \in E_\eta,$ and let $P:[0,1]\to \R^d$ be a $e^f\cdot D_0$-distance minimizing path between them. Then if $C_0,C_1,\bar{n},\bar{m},\eta$ are chosen such that \eqref{param1}, \eqref{param2}, \eqref{param3} hold, we have
\[
P([0,1]) \subset E_{\eta+ \frac{\bar{\eta}}{2}}.
\]
\end{lemma}
These lemmas will be proved in subsection \ref{appproofs}. For now we will assume them and prove \eqref{closemetrics}.

We will proceed by proving upper and lower bounds for $e^f \cdot D_0(x,y)$ in the case that $x,y\in E_\eta$ first in subsection \ref{edges}, and then in subsection \ref{finalappproof} we prove upper and lower bounds in the case that $x,y$ are far from any edge.

\subsection{Comparing metrics: the edge case}\label{edges}

The main results in this subsection are the following.

\begin{lemma}\label{upperedge}
Suppose that $x,y$ are vertices on the graph $G.$ Then we have the inequality
\begin{equation}
e^f\cdot D_0(x,y) \leq d_G(x,y).
\end{equation}
\end{lemma}
\begin{lemma}\label{loweredge}
Suppose that $x,y$ are vertices in $G.$ Then we have the bound
\begin{equation}
e^f\cdot D_0(x,y) \geq \bar{D}(x,y)-\frac{\e}{4}.
\end{equation}
\end{lemma}

First we prove Lemma \ref{upperedge}.
\begin{proof}[Proof of Lemma \ref{upperedge}]
Let $\{e_i\}_{i=1}^n$ be a path consisting of edges on $G$ connecting $x,y$ with minimal $d_G$-length. Let $\{x_i\}_{i=1}^n$ be the vertices that this path passes through in order, so that $x_i$ and $x_{i+1}$ are adjacent, and $x_1=x,$ $x_n=y.$ Then we have by the triangle inequality that
\[
e^f\cdot D_0(x,y) \leq \sum_{i=1}^{n-1} e^f\cdot D_0(x_i,x_{i+1}).
\]
Now using Lemma \ref{graphapproxupperbound} we obtain that
\[
e^f\cdot D_0(x,y) \leq \sum_{i=1}^{n-1} d_G(x_i,x_{i+1}) = d_G(x,y).
\]
This proves Lemma \ref{upperedge}. 
\end{proof}
We will now prove Lemma \ref{loweredge}.

Let $P:[0,1]\to \R^d$ be a $e^f \cdot D_0$-length minimizing path between $x,y$ such that $P(0)=x,$ $P(1)=y.$ Then we know by Lemma \ref{trapped} that $P([0,1])\subset E_{\eta+\bar{\eta}}.$ We must prove that for any vertices $x,y,$
\[
d_G(x,y)-\frac{\e}{4} \leq \ell_{e^f\cdot D_0}(P).
\]
For each edge $e$ on $G$ and $\theta>0$ we let $T_e^\theta:=\{z: D_0(z,e)\leq \theta\}.$ Let $E_{x,y}$ be the minimal set of edges such that
\[
\cup_{e \in E_{x,y}} T_e^{\eta+\bar{\eta}}
\]
is connected, and contains $P([0,1]).$ Let $\{x_i\}_{i=1}^n$ be such that $x_i$ and $x_{i+1}$ are adjacent, the edge between $x_i,x_{i+1}$ is in $E_{x,y},$ $x\in T_{(x_1,x_2)}^{\eta+\bar{\eta}},$ and $y\in T_{(x_{n-1},x_n)}^{\eta+\bar{\eta}},$ where $(v_1,v_2)$ denotes the edge between $v_1,v_2$ if $v_1,v_2$ are adjacent vertices in $G$ (see Figure \ref{xipic}).

\begin{figure}
\begin{center}
\includegraphics[width=0.35\textwidth]{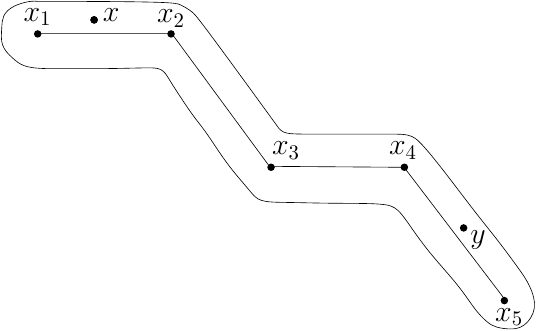}  
\caption{\label{xipic} The points $x_i.$
}
\end{center}
\end{figure}

Now we need the following lemma.
\begin{lemma}\label{refinedloweredge}
We have that
\begin{equation}
e^f\cdot D_0(x_1,x_{n-1}) \geq \sum_{i=1}^{n-1} \bar{D}(x_i,x_{i+1}) - \frac{\e}{20}.
\end{equation}
\end{lemma}
\begin{proof}
Let $\tilde{P}:[0,1]\to \R^d$ be an $e^f \cdot D_0$-length minimizing path between $x_1,x_{n-1}.$ Recall that by Lemma \ref{trapped} we know that $\tilde{P}$ is contained in $E_{\eta+\frac{\bar{\eta}}{2}}.$ Suppose that $e_0$ is an edge and $x \in T_{e_0}^\theta \setminus \cup_{e \neq e_0} T_{e}^\theta$ for small enough $\theta>0,$ and $\bar{x}$ is such that $D_0(x,\bar{x}) = D_0(x,e).$ Then if $\bar{x}$ is at $D_0$-distance at least $\theta$ from the endpoints of $e_0,$ we have $f(\bar{x})\leq f(x)$ (this comes directly from the definition of $f$). We will define the sequences of times $\{s_i\}_{i=1}^n, \{\bar{s}_i\}_{i=1}^n$ such that
\[
s_i=\sup\{t:\tilde{P}(t)\in B_{\eta+\bar{\eta}}(x_i)\},
\]
and
\[
\bar{s}_i=\inf\{t:\tilde{P}(t)\in B_{\eta+\bar{\eta}}(x_{i+1})\},
\]
so that
\[
s_1<\bar{s}_1<s_2 < \cdots < s_n<\bar{s}_n.
\]
Then using that for any edges $e_1,e_2$ we have that $\mathrm{Diam}_d(T_{e_1}^{\eta+\bar{\eta}} \cap T_{e_2}^{\eta+\bar{\eta}}) \to 0$ as $\eta,\bar{\eta}\to 0$ it suffices to prove that
\[
\ell_{e^f\cdot D_0}(\tilde{P}\vert_{[s_i,s_{i+1}]}) \geq d(x_i,x_{i+1}) - \frac{\e}{20\vert G\vert},
\]
where $\vert G\vert$ denotes the number of edges in the graph $G.$ Let $(x_i,x_{i+1})$ be the segment between $x_i$ and $x_{i+1},$ and let $\bar{P}(t)$ be such that if $t \in [s_i,s_{i+1}]$ we have
\[
\vert \bar{P}(t) - \tilde{P}
(t)\vert = D_0(\tilde{P}(t),(x_i,x_{i+1})).
\]
Then
\[
\ell_{e^f\cdot D_0}(\tilde{P}\vert_{[s_i,s_{i+1}]}) = \int_{s_i}^{s_{i+1}} e^{f(\tilde{P}(t))} \vert \tilde{P}'(t)\vert dt \geq \int_{s_i}^{s_{i+1}} e^{f(\bar{P}(t))} \vert \bar{P}'(t)\vert,
\]
Then for small enough $\eta,\bar{\eta}$
\[
\ell_{e^f\cdot D_0}(\tilde{P}\vert_{[s_i,s_{i+1}]}) \geq \ell_{e^f\cdot D_0}(\bar{P}\vert_{[s_i,s_{i+1}]})-\frac{\e}{20\vert G\vert} \geq d(x_i,x_{i+1})-\frac{\e}{20\vert G\vert}.
\]
Noting that
\[
\ell_{e^f\cdot D_0}(\bar{P}) = \sum_{i=1}^{n-1} \ell_{e^f\cdot D_0}(\bar{P}\vert_{[s_i,s_{i+1}]})
\]
we obtain Lemma \ref{refinedloweredge}.
\end{proof}

Now we will prove Lemma \ref{loweredge}.
\begin{proof}[Proof of Lemma \ref{loweredge}]
By the triangle inequality, we have that
\begin{equation}\label{triangle3}
\bar{D}(x,y) \leq \bar{D}(x,x_1)+\bar{D}(x_n,y) + \sum_{i=1}^{n-1} \bar{D}(x_i,x_{i+1}).
\end{equation}
Note that $\mathrm{Diam}_{d}(T_e^\theta)$ tends to $\mathrm{Diam}_d(e)$ as $\theta \to 0.$ Letting $\delta$ be small enough we obtain that since $x\in T_{(x_1,x_2)}^{\eta+\bar{\eta}},$ if $\eta$ and $\bar{\eta}$ are small enough we know that
\[
\bar{D}(x,x_1) \leq \frac{\e}{20},
\]
and similarly
\[
\bar{D}(y,x_{n-1}) \leq \frac{\e}{20}.
\]
Combining this with \eqref{triangle3} we obtain that
\[
\bar{D}(x,y) \leq \sum_{i=1}^{n-1} \bar{D}(x_i,x_{i+1}) + \frac{\e}{10}.
\]
Similarly we obtain by the triangle inequality that
\begin{eqnarray*}
e^f \cdot D_0(x,y)&\geq& e^f\cdot D_0(x_1,x_{n-1}) - e^f\cdot D_0(x,x_1) - e^f\cdot D_0(y,x_{n-1})\\
&\geq& e^f\cdot D_0(x_1,x_{n-1}) -\frac{\e}{10}.
\end{eqnarray*}
Combining this with Lemma \ref{refinedloweredge} we conclude the proof of Lemma \ref{loweredge}.
\end{proof}

\subsection{Proof of Theorem \ref{app}}\label{finalappproof}
In this section we prove Theorem \ref{app}. We will need the following lemma.
\begin{lemma}\label{triptohighway}
Suppose that $x \in [-R,R]^d.$ Then there exists a vertex $\bar{x} \in \mathcal{V}\mathcal{G}$ such that
\[
\bar{D}(x,\bar{x}),e^f \cdot D_0(x,\bar{x}) \leq \frac{\e}{16}
\]
\end{lemma}
We assume Lemma \ref{triptohighway} for now, and prove Theorem \ref{app}.
\begin{proof}[Proof of Theorem \ref{app}]
Let $x,y \in [-R,R]^d.$ Let $\bar{x},\bar{y}$ be as in Lemma \ref{triptohighway}. Then
\begin{eqnarray*}
\bar{D}(x,y) & \leq & \bar{D}(x,\bar{x})+\bar{D}(y,\bar{y}) + \bar{D}(\bar{x},\bar{y})\\
&\leq & \bar{D}(\bar{x},\bar{y}) + \frac{\e}{8}.
\end{eqnarray*}
By Lemma \ref{loweredge} we have
\[
\bar{D}(\bar{x},\bar{y})\lesssim e^f\cdot D_0(\bar{x},\bar{y})+ \frac{\e}{4}.
\]
Hence
\begin{eqnarray*}
\bar{D}(x,y) & \leq & e^f\cdot D_0(\bar{x},\bar{y})+\frac{\e}{4}+\frac{\e}{8}\\
&\leq & e^f\cdot D_0(x,y) + e^f\cdot D_0(x,\bar{x})+e^f\cdot D_0(y,\bar{y}) + \frac{\e}{4}+\frac{\e}{8} \\
&\leq & e^f\cdot D_0(x,y) + \frac{\e}{2}.
\end{eqnarray*}
On the other hand, we have
\begin{eqnarray*}
e^f\cdot D_0(x,y) &\leq& e^f\cdot D_0(\bar{x},\bar{y}) + e^f\cdot D_0(x,\bar{x})+e^f\cdot D_0(y,\bar{y})\\
&\leq& e^f\cdot D_0(\bar{x},\bar{y}) + \frac{\e}{4}\\
&\leq & d_G(\bar{x},\bar{y})+\frac{\e}{4}
\end{eqnarray*}
where in the last line we used Lemma \ref{upperedge}. Now by \eqref{dGdclose} we obtain
\[
e^f\cdot D_0(x,y) \leq d_G(\bar{x},\bar{y})+\frac{\e}{4} \leq \bar{D}(\bar{x},\bar{y}) + \frac{\e}{2}.
\]
This completes the proof.
\end{proof}
Now we will prove Lemma \ref{triptohighway}.
\begin{proof}[Proof of Lemma \ref{triptohighway}.]
We will divide the proof in three steps.\\
\underline{Step 1:} Suppose there exists an edge $e$ such that $D_0(x,e)\leq \eta-\bar{\eta}.$ Then there exists a vertex $\bar{x}$ that is one of $e$'s endpoints such that $\bar{D}(x,\bar{x}),e^f \cdot D_0(x,\bar{x}) \leq \frac{\e}{32}.$ Then choosing $\tau$ to be small enough and $K$ large enough, we obtain
\[
\bar{D}(x,y_0) \leq \bar{D}(y_0,y_1)+\frac{\e}{64} \leq \frac{\e}{64}+\frac{\e}{64} = \frac{\e}{32}.
\]
Similarly, we have
\[
e^f\cdot D_0(x,y_0) \leq \frac{w_e}{\ell_{D_0}(e)} D_0(x,y_0) \leq D_0(x,y_0) \leq \frac{\e}{32}.
\]
This completes step 1.\\
\underline{Step 2:} Suppose there is an edge $e$ such that $D_0(x,e)\leq \eta+\bar{\eta}.$ Then there exists a vertex $\bar{x}$ that is one of $e$'s endpoints such that $\bar{D}(x,\bar{x}),e^f \cdot D_0(x,\bar{x}) \leq \frac{3\e}{64}.$ Note that by step 1, it suffices to show that there exists a point $\bar{x}_1 \in E_{\eta-\bar{\eta}}$ such that
\[
e^f\cdot D_0(x,\bar{x}_1), \bar{D}(x,\bar{x}_1) \leq \frac{3\e}{64}.
\]
For this, let $\bar{x}_1$ be the point on $E_{\eta-\bar{\eta}}$ such that $D_0(x,\bar{x}_1)$ is minimized. Then by \eqref{param2},
\[
e^f \cdot D_0(x,\bar{x}_1) \leq e^{C_0} 2\bar{\eta} = \frac{\e}{64}.
\]
Additionally, note that for small enough $\eta,\bar{\eta},$ we have $\bar{D}(x,\bar{x}_1) \leq \frac{\e}{64}$ by continuity of the metric $\bar{D}$ with respect to $D_0.$ This completes this case.\\
\underline{Step 3:} Conclusion of the proof. By steps 1 and 2, the result holds if $x \in E_{\eta+\bar{\eta}},$ thus we suppose that $x \in \R^d \setminus E_{\eta+\bar{\eta}}.$ Let $i_1, \ldots i_d$ be the unique indices such that $x \in S_{i_1\cdots i_d},$ and let $e$ be an edge contained in $S_{i_1\cdots i_d}$ in the same connected component of $S_{i_1\cdots i_d} \setminus E_{\eta+\bar{\eta}}.$ Recall by \eqref{smalldiam}, if we let $\bar{x}$ be an endpoint of $e,$ then $\bar{D}(x,\bar{x}) \leq \frac{\e}{16}.$ On the other hand, to show that $e^f\cdot D_0(x,\bar{x}) \leq \frac{\e}{16},$ by step 2 it suffices to show that there exists an $\bar{x}_2\in \{z: D_0(z,e)\leq \eta+\bar{\eta}\}$ such that $e^f\cdot D_0(x,\bar{x}_2) \leq \frac{\e}{64}.$ To show this, we take $\bar{x}_2$ to be the point such that $D_0(x,\bar{x}_2)$ is minimized, with the constraint $D_0(\bar{x}_2,e)\leq \eta+\bar{\eta}.$ Then by definition of $f,$ we have by \eqref{param3}
\[
e^f\cdot D_0(x,\bar{x}_2) \leq \frac{1}{\bar{n}\bar{m}} e^{C_1} = \frac{\e}{64}.
\]
This completes the proof.
\end{proof}

\subsection{Proofs of Lemmas \ref{graphapproxupperbound}, \ref{graphapproxlowerbound} and \ref{trapped}}\label{appproofs}
We will now prove Lemmas \ref{graphapproxupperbound}, \ref{graphapproxlowerbound} and \ref{trapped}. We start with the proof of Lemma \ref{graphapproxupperbound}.
\begin{proof}[Proof of Lemma \ref{graphapproxupperbound}] Note that if $v_1,v_2$ are adjacent edges, then if $P$ denotes the straight line segment between $v_1,v_2$ we have
\[
e^f\cdot D_0(v_1,v_2) \leq \ell_{e^f\cdot D_0}(P) \leq w_{(v_1,v_2)} = d_G(v_1,v_2) = d(v_1,v_2)
\]
by construction of $G.$ This completes the proof.
\end{proof}
Now we will prove Lemma \ref{trapped}.

\begin{proof}[Proof of Lemma \ref{trapped}]
Suppose that $x,y \in E_\eta,$ and $P$ is a $e^f\cdot D_0$-distance minimizing path between them with $P(0)=x,$ $P(1)=y.$ The idea of the proof is that if $P$ exits $E_{\eta+\frac{\bar{\eta}}{2}}$ then $P$ can be replaced by a path with smaller $e^f\cdot D_0$-length, which would yield a contradiction.

Suppose that $P$ exits $E_{\eta+\frac{\bar{\eta}}{2}}.$  Suppose that $t_{\mathrm{start}},t_{\mathrm{end}}$ are chosen so that
\[
P([t_{\mathrm{start}},t_{\mathrm{end}}]) \subset \R^d\setminus E_\eta, \; P(t_{\mathrm{start}}),P(t_{\mathrm{end}}) \in \p E_\eta, \; P([t_{\mathrm{start}},t_{\mathrm{end}}]) \cap \R^d \setminus E_{\eta+\frac{\bar{\eta}}{2}} \neq \varnothing.
\]
That is, $t_{\mathrm{start}},t_{\mathrm{end}}$ are the entrance and exit times of an excursion, chosen so that $P$ stays outside $E_\eta$ between these two times. Let
\[
\ell = \ell_{D_0}(P\vert_{[t_{\mathrm{start}},t_{\mathrm{end}}]}).
\]
Then we claim that we can define a path $\tilde{P}:[0,1] \to \R^d$ such that $\ell_{e^f\cdot D_0} (\tilde{P}) < \ell_{e^f\cdot D_0}(P\vert_{[t_{\mathrm{start}},t_{\mathrm{end}}]})$ and $\tilde{P}(0)=P(t_{\mathrm{start}}), \tilde{P}(1)=P(t_{\mathrm{end}}).$ Note that this would yield a contradiction, since the concatenation $P\vert_{[0,t_{\mathrm{start}}]} \ast \tilde{P}\ast P\vert_{[t_{\mathrm{end}},1]}$ would have a smaller $e^f\cdot D_0$ length than $P,$ which is impossible since $P$ is an $e^f\cdot D_0$-distance minimizing path.

To define $\tilde{P},$ recall that $P(t_{\mathrm{start}}),P(t_{\mathrm{end}}) \in \p E_\eta,$ and so there exist points (not necessarily vertices) $\bar{v}_{\mathrm{start}},\bar{v}_{\mathrm{end}} \in E$ such that
\begin{equation}\label{eqidk}
e^f\cdot D_0(P(t_{\mathrm{start}}),\bar{v}_{\mathrm{start}}),e^f\cdot D_0(P(t_{\mathrm{end}}),\bar{v}_{\mathrm{end}}) \leq \eta e^{\norm{f}_{L^\infty(E_\eta)}}.
\end{equation}
Note that this is possible, since the Euclidean distance between any point in $\p E_\eta$ and $E$ is at most $\eta.$ Let $v_{\mathrm{start}},v_{\mathrm{end}}$ denote the vertices of $G$ nearest to $\bar{v}_{\mathrm{start}},\bar{v}_{\mathrm{end}}$ respectively with respect to the Euclidean metric. Then we have by Lemma \ref{upperedge} that
\[
e^f\cdot D_0(\bar{v}_{\mathrm{start}},\bar{v}_{\mathrm{end}}) \leq \max_{e\in \mathcal{E}G} \ell_{e^f\cdot D_0}(e) + e^f\cdot D_0(v_{\mathrm{start}},v_{\mathrm{end}}) \leq \max_{e\in \mathcal{E}G} \ell_{e^f\cdot D_0}(e)+\bar{D}(v_{\mathrm{start}},v_{\mathrm{end}}).
\]
We are now ready to define $\tilde{P}.$ Let $\tilde{P}_{\mathrm{start}}$ be an $e^f\cdot D_0$-distance minimizing path between $P(t_{\mathrm{start}}),\bar{v}_{\mathrm{start}},$ and similarly let $\tilde{P}_{\mathrm{end}
}$ be an $e^f\cdot D_0$-distance minimizing path between $\bar{v}_{\mathrm{end}},P(t_{\mathrm{end}}).$ Let $\tilde{P}_{\mathrm{vstart}}$ denote the shortest path contained in $E$ between $\bar{v}_{\mathrm{start}}$ and $v_{\mathrm{start}},$ and $\tilde{P}_{\mathrm{vend}}$ the shortest path contained in $E$ between $v_{\mathrm{end}}$ and $\bar{v}_{\mathrm{end}}.$ Finally, we let $\tilde{P}_{\mathrm{inner}}$ denote the $d_G$-distance minimizing path between $v_{\mathrm{start}},v_{\mathrm{end}}.$ We define $\tilde{P}$ to be the concatenation
\[
\tilde{P} := \tilde{P}_{\mathrm{start}} \ast \tilde{P}_{\mathrm{vstart}} \ast \tilde{P}_{\mathrm{inner}} \ast \tilde{P}_{\mathrm{vend}} \ast \tilde{P}_{\mathrm{end}}.
\]
Note that by \eqref{eqidk} we have that
\[
\ell_{e^f\cdot D_0}(\tilde{P}_{\mathrm{start}}),\ell_{e^f\cdot D_0}(\tilde{P}_{\mathrm{end}}) \leq \eta e^{\norm{f}_{L^\infty(E_\eta)}}.
\]
Also,
\[
\ell_{e^f\cdot D_0}(\tilde{P}_{\mathrm{vstart}}), \ell_{e^f\cdot D_0}(\tilde{P}_{\mathrm{vend}}) \leq \frac{1}{2}\max_{e \;\mathrm{edge}}\ell_{e^f\cdot D_0}(e) \leq \frac{1}{2}\max_{e \;\mathrm{edge}} \ell_{D_0}(e) \frac{w_e}{\ell_{D_0}(e)} =\frac{1}{2}\max_{e \; \mathrm{edge}} w_e.
\]
Then we compute
\begin{eqnarray*}
\ell_{e^f\cdot D_0}(\tilde{P}) &\leq& \ell_{e^f\cdot D_0}(\tilde{P}_{\mathrm{start}}) + \ell_{e^f\cdot D_0}(\tilde{P}_{\mathrm{vstart}}) + \ell_{e^f\cdot D_0}(\tilde{P}_{\mathrm{inner}}) + \ell_{e^f\cdot D_0}(\tilde{P}_{\mathrm{vend}}) + \ell_{e^f\cdot D_0}(\tilde{P}_{\mathrm{end}})\\
&\leq & 2\eta e^{\norm{f}_{L^\infty(E_\eta)}} + \ell_{e^f\cdot D_0}(\tilde{P}_{\mathrm{vstart}}) + \ell_{e^f\cdot D_0}(\tilde{P}_{\mathrm{inner}}) + \ell_{e^f\cdot D_0}(\tilde{P}_{\mathrm{vend}})\\
&\leq & 2\eta e^{\norm{f}_{L^\infty(E_\eta)}} + \max_{e \;\mathrm{edge}} w_e + \ell_{e^f\cdot D_0}(\tilde{P}_{\mathrm{inner}})\\
&\leq & 2\eta e^{\norm{f}_{L^\infty(E_\eta)}} + \max_{e \;\mathrm{edge}} w_e + \vphi(\vert v_{\mathrm{start}}-v_{\mathrm{end}} \vert)
\end{eqnarray*}
where in the last line we used that
\[
\ell_{e^f\cdot D_0}(\tilde{P}_{\mathrm{inner}}) \leq \vphi(\vert v_{\mathrm{start}}-v_{\mathrm{end}} \vert)
\]
where $\vphi$ is as defined in \eqref{vphidefdef} which follows from Lemma \ref{graphapproxupperbound}. By defintion of $f,$ we have
\[
\norm{f}_{L^\infty(E_\eta)} \leq \max_{e \; \mathrm{edge}}\log\left(\frac{2w_e}{\ell_{D_0}(e)}\right).
\]
Now combining this with \eqref{etasmol}, \eqref{fdef}, \eqref{smallwe} we have that
\[
2\eta e^{\norm{f}_{L^\infty(E_\eta)}} + \max_{e \;\mathrm{edge}} w_e  \leq \max_{e \;\mathrm{edge}}w_e + \frac{\e}{128} < \frac{\e}{64} = 2e^{C_0}\bar{\eta}
\]
where we used \eqref{param2} in the last inequality. Now note that if $\ell \geq \Psi(\e),$ then
\[
\vphi(\vert v_{\mathrm{start}}-v_{\mathrm{end}} \vert) \leq e^{C_1}(\ell-2\bar{\eta})
\]
since
\[
e^{C_1} = \frac{\e\bar{n}\bar{m}}{64} \geq \sup_{\Psi(\e) \leq \ell \leq 2R} \frac{\vphi(\ell)}{\ell-2\bar{\eta}}
\]
by \eqref{param1} and \eqref{param3}. Therefore we obtain
\begin{equation}\label{firsthand}
\ell_{e^f\cdot D_0}(\tilde{P}) < 2e^{C_0}\bar{\eta} + e^{C_1}(\ell-2\bar{\eta}).
\end{equation}
On the other hand,
\begin{eqnarray}\label{contradiction}
\ell_{e^f\cdot D_0}(P\vert_{[t_{\mathrm{start}},t_{\mathrm{end}}]}) &\geq&  e^{C_0}\ell_{D_0}(P\vert_{[t_{\mathrm{start}},t_{\mathrm{end}}]}\cap E_{\eta+\frac{\bar{\eta}}{2}}) + e^{C_1}\left(\ell-\ell_{D_0}(P\vert_{[t_{\mathrm{start}},t_{\mathrm{end}}]}\cap E_{\eta+\frac{\bar{\eta}}{2}})\right)\nonumber\\
&\geq& \left(e^{C_0}-e^{C_1}\right) \ell_{D_0}(P\vert_{[t_{\mathrm{start}},t_{\mathrm{end}}]}\cap E_{\eta+\frac{\bar{\eta}}{2}}) + e^{C_1}\ell \nonumber\\
&\geq & \left(e^{C_0}-e^{C_1}\right)\bar{\eta} + e^{C_1}\ell \nonumber\\
&=& e^{C_0} \bar{\eta}+ e^{C_1}(\ell-\bar{\eta}).
\end{eqnarray}
Now combining \eqref{firsthand} and \eqref{contradiction} we obtain a contradiction.
Now in the case that $\ell \leq \Psi(\e),$ we have that $\vphi(\vert v_{\mathrm{start}}-v_{\mathrm{end}}\vert) \leq \e$ and we again obtain a contradiction. This completes the proof.

\end{proof}

Finally, we prove prove Lemma \ref{graphapproxlowerbound}.
\begin{proof}[Proof of Lemma \ref{graphapproxlowerbound}]
Suppose that $p:[0,1]\to \R^d$ is a $e^f\cdot D_0$-distance minimizing path between $v_1,v_2.$ Note that by Lemma \ref{trapped}, $p([0,1]) \subseteq (v_1,v_2)+B_{\eta+\bar{\eta}}(0),$ where again we recall that $(v_1,v_2)$ denotes the edge between $v_1,v_2.$ We claim that
\[
\ell_{e^f\cdot D_0}(p) \geq d_G(v_1,v_2)-\frac{\e}{2}.
\]
Let $P$ be the projection of the path $p$ onto the edge $e,$ that is for each $t\in [0,1],$ let $P(t)$ be the unique point such that $\mathrm{Dist}(P(t),p(t))=\inf_{z\in e}\mathrm{Dist}(z,p(t)).$ It is easy to check that $P(t)$ is also continuous. For $x_0 \in \R^d$, let $B_r^d(x_0)$ denote the $d$-ball of radius $r$ centered at $x_0,$ and let $t_0,t_1$ be defined by
\[
t_0=\sup\{t:P(t) \in B_{\tau}^d(v_1)\},
\]
\[
t_1=\inf\{t:P(t)\in B_{\tau}(v_2)\}.
\]
$t_0,t_1$ are used to split off the endpoints of the edge. This way, $f$ is smaller at the projection onto the edge if away from the endpoints. We thus have for small enough $\tau$ that
\[
\ell_{e^f\cdot D_0}(p) = \int_0^1 e^{f(p(t))} \vert p'(t)\vert dt\geq \int_{t_0}^{t_1} e^{f(p(t))} \vert p'(t)\vert dt - 2\tau \geq \int_{t_0}^{t_1} F(P(t)) \vert P'(t)\vert dt \geq \ell_e^{f(P)}-4\tau=w-4\tau.
\]
This completes the proof. 
\end{proof}

\bibliography{ref,cibib}
\bibliographystyle{hmralphaabbrv}

\end{document}